\documentclass{amsart}

\usepackage{wasysym}

\usepackage[mathscr]{eucal}

\usepackage{amsmath}
\usepackage{stmaryrd}
\usepackage{amssymb}
\usepackage{amsthm}

\usepackage{amsfonts}
\usepackage[pdftex]{graphicx}
\usepackage{tikz}
\usepackage{enumerate}
\usepackage{subfigure}

\usepackage[markup=colored]{changes}
\setauthormarkuptext{name}
\definechangesauthor[name={JRL}, color=blue]{j}

\newtheorem{theorem}{Theorem}[section]
\newtheorem{lemma}[theorem]{Lemma}
\newtheorem{proposition}[theorem]{Proposition}
\newtheorem{corollary}[theorem]{Corollary}

\newtheorem{problem}[theorem]{Problem}

\theoremstyle{definition}
\newtheorem{definition}[theorem]{Definition}
\newtheorem{example}[theorem]{Example}

\newtheorem{remark}[theorem]{Remark}

\begin{document}
	
	\title[Strongly quasi-pseudometric aggregation functions]{Strongly quasi-pseudometric aggregation functions}
	
	\author[A. Fructuoso-Bonet]{Alejandro Fructuoso-Bonet}
	\address[A. Fructuoso-Bonet]{Programa de Doctorado en Matem\'aticas, Instituto Universitario de Matem\'atica Pura y Aplicada, Universitat Polit\`ecnica de Val\`encia, Camino de Vera s/n, 46022 Valencia, Spain,\newline    Florida Universitaria, C/ Rei en Jaume I, nº 2, 46470 Catarroja (Valencia), Spain}
	
	\email{afrubon@posgrado.upv.es}
	\author[J. Rodr\'{\i}guez-L\'opez]{Jes\'us Rodr\'{\i}guez-L\'opez}
	\address[J. Rodr\'{\i}guez-L\'opez]{Instituto Universitario de Matem\'atica Pura y Aplicada, Universitat Polit\`ecnica de Val\`encia, Camino de Vera s/n, 46022 Valencia, Spain}
	\thanks{Last author's research
		is part of the project PID2022-139248NB-I00 funded by MICIU/AEI/10.13039/501100011033 and ERDF/EU}
	\email{jrlopez@mat.upv.es}

	

	\begin{abstract}
		Metric-preserving functions (here, metric aggregation functions) offer a natural method for constructing metrics on Cartesian products of metric spaces or for aggregating multiple metrics defined on a common set. Strongly metric-preserving functions represent a more specialized subset of these functions, ensuring that the new metric aligns with the product topology, in the Cartesian product case.  However, these strong functions have not been previously explored for quasi-pseudometrics. Furthermore, in the case where all metrics are defined on the same set, the problem has not been addressed previously.
		
		In this paper, we investigate the class of strongly (quasi-)(pseudo)metric aggregation functions, extending the classical concept.  We begin by examining the case where the aggregation function produces (quasi-)(pseudo)metrics on Cartesian products, characterizing these functions through continuity at zero and a minimal zero preimage condition.
		
		In addition, we will examine the scenario where the aggregation function produces a (quasi-)(pseudo)metric defined on a fixed set. Within this context, we will demonstrate that the appropriate topology to consider is the supremum topology. We will also provide both necessary and sufficient conditions for an (quasi-)(pseudo)metric aggregation function on sets to qualify as a strongly one, thereby addressing a gap in the existing literature.
		
	\end{abstract}
	
	\date{\today}
	\maketitle

	%
	\section{Introduction}
	%
	
	When a new mathematical space is introduced, one fundamental question that arises is how it interacts with classical set-theoretic operations, such as the Cartesian product.  For instance, given a family of metric spaces $\big\{(X_i,d_i):i\in I\big\},$ one may ask how to endow the product $\prod_{i\in I} X_i$ with a metric. 
	This question becomes trivial if we ignore the metrics $d_i$, as one could always use the discrete metric. However, it’s essential to move beyond this naive approach and explore methods for defining a metric on $\prod_{i \in I} X_i$ that genuinely considers the metrics of the individual factor spaces.

	A general approach to this problem was developed by Bors\'ik and Dobo\v{s} \cite{BorsikDobos81a,Dobos98}. They analyzed functions $F:[0,+\infty)^I\to [0,+\infty)$ such that, for any family of metric spaces $\big\{(X_i,d_i):i\in I\big\},$  the composition $F\circ d_\Pi$ defines a metric on $\prod_{i\in I} X_i.$ Here, the mapping $d_\Pi:\prod_{i\in I} X_i\times \prod_{i\in I} X_i\to [0,+\infty)^I$ is defined by
	$$d_\Pi\big((x_i)_{i\in I},(y_i)_{i\in I}\big)=\big(d_i(x_i,y_i)\big)_{i\in I}$$
	for all $(x_i)_{i\in I},(y_i)_{i\in I}\in\prod_{i\in I} X_i.$ 
	
	These functions were originally called \emph{metric preserving functions} and where characterized in \cite{BorsikDobos81b,Dobos98}. Specifically, it was shown that  $F$ is metric preserving if and only if whenever $(a,b,c)\in ((0,+\infty)^I)^3$ satisfies $a_i\leq b_i+c_i, $ $b_i\leq a_i+c_i, $ $c_i\leq a_i+b_i, $ for every $i\in I$ ($(a,b,c)$ is a \emph{triangle triplet}), then $\big(F(a),F(b),F(c)\big)$ is also a triangle triplet. Since these functions allow for a metric to be obtained on $\prod_{i \in I} X_i$ by, in some sense, aggregating the metrics of the family $\{d_i : i \in I\},$ contemporary literature often incorporates the adjective ``aggregation'' to refer to these functions.

	The origins of this problem can be traced back to a classical paper by Wilson \cite{Wilson35}, where certain types of continuous transformations between metric spaces were explored. One of these transformations was defined using a so-called scale function, which is essentially a metric preserving function of one variable.
	
	Later, a more detailed study was conducted by Sreenivasan \cite{Sree47}, again focusing on one-variable functions $f:[0,+\infty)\to [0,+\infty)$ (for a survey of these functions, see \cite{Corazza99}). We emphasize that aggregation functions have numerous applications, particularly in decision-making problems (see, for instance, \cite{Pap15,BookTorraNaru, LHSR22,Demir23}).  Moreover, metric preserving functions have been employed to develop flexible models for index spaces \cite{ArnauCalabuigGonzalezSP24}.
	
	Sreenivasan raised the following more interesting question: given a metric preserving function $f:[0,+\infty)\to [0,+\infty)$ and a metric space $(X,d),$ is the metric $d$ topologically equivalent to $f\circ d$? In general, the answer is negative, and functions that satisfy this property are referred to as \emph{strongly metric preserving} as described in \cite{Corazza99,Dobos98}. These functions where characterized in \cite{BorsikDobos81a} by means of continuity of $F$ at $0.$
	
	Additionally, the same authors characterized strongly metric preserving functions of several variables in \cite{BorsikDobos81b}. To clarify this notion, let \( F: [0, +\infty)^I \to [0, +\infty) \) be a metric preserving function. \( F \) is considered \emph{strongly metric preserving} if, for any collection of metric spaces $\big\{(X_i,d_i):i\in I\big\},$ the product topology $\prod_{i\in I}\mathscr{T}(d_i)$ is equal to the topology $\mathscr{T}(F\circ d_\Pi)$ generated by the metric $F\circ d_\Pi$. This property is equivalent to the continuity of $F$ at $(0)_{i\in I}.$
	
	The notion of metric preserving functions has also been extended to the asymmetric context by Mayor and Valero \cite{MayorVale10} by considering families of quasi-metric spaces instead of traditional metric spaces. Characterizations for pseudometric or quasi-pseudometric preserving functions can be found in \cite{PraTri02} and \cite{FBRL25a}, respectively. However, to the best of our knowledge, the characterization of strongly quasi-pseudometric preserving functions has not yet been explored in the literature. Therefore, the primary objective of this paper is to address this gap by providing the necessary characterizations.
	
	Furthermore, we investigate another related problem that appears to have been overlooked. Suppose that we are giving a family of metrics $\{d_i:i\in I\}$ all defined on the same underlying set $X.$ Can a new metric on $X$ be defined through a function $F: [0, +\infty)^I \to [0, +\infty)$? In this context, we can define $d_\Delta: X\times X\to [0,+\infty)^I$ as
	$$d_\Delta(x,y)=\big(d_i(x,y)\big)_{i\in I}$$
	for every $x,y\in X.$ The question then arises as to whether $F\circ d_\Delta$ defines a metric on $X.$ The characterization of functions that satisfy this property has been established by Mayor and Valero \cite{MayorVale19}, although this issue had previously been examined for pseudometrics by Praderas and Trillas \cite{PraTri02}. Mi\~nana and Valero also tackled the problem for quasi-metrics \cite{MinyaVale19}. None of these studies, however, addressed the corresponding ``strong'' version of the problem, which takes into account the induced topology. In this paper, we also explore this issue by clarifying what the strongly preserving property should mean in this context.

	We conclude this introduction with a brief summary of the paper.

	In Section \ref{sec2}, we explicitly introduce the definition of a (quasi-)(pseudo)metric aggregation function and distinguish between two different aggregation frameworks. The first one, which we refer to as aggregation \emph{on products}, follows the classical approach introduced Bors\'ik and Dobo\v{s}, where the problem of constructing a metric on the Cartesian product of metric spaces is considered (in their terminology, these functions are called metric-preserving functions). The second framework, which we call aggregation \emph{on sets}, is based on the approach developed by Pradera and Trillas \cite{PraTri02} and Mayor, Mi\~nana and Valero \cite{MayorVale19,MinyaVale19}. This latter approach studies functions that merge a family of metrics defined on a common underlying set into a single metric. We also recall several known results concerning the characterization of these aggregation functions.
	
	Section \ref{sec3} addresses the problem of characterizing strongly quasi-pseudometric aggregation functions on products. Our main result is Theorem \ref{thm:ChactStronglyProd}, which provides a complete characterization of strongly (quasi-)(pseudo)metric aggregation functions on products. More precisely, we show that the class of strongly quasi-metric aggregation functions on products coincides with that of strongly quasi-pseudometric aggregation functions on products. Furthermore, this common class forms a subclass of the strongly metric aggregation functions, which coincides with the class of strongly pseudometric aggregation functions. Our characterization is also expressed in terms of well-known properties.
	
	The final section, Section \ref{sec4}, is dedicated to exploring the preservation of topology in the context of (quasi-)(pseudo)metric aggregation functions on sets. In this framework, the supremum topology serves as the equivalent of the product topology. A key result from this section is Theorem \ref{thm:ChactStronglySets}, where we demonstrate that strongly aggregation on products coincides with strongly aggregation on sets for quasi-pseudometrics, quasi-metrics, and pseudometrics. Furthermore, we offer a characterization of strongly metric aggregation functions on sets, as detailed in Theorem \ref{thm:smafs}. However, we remain uncertain whether these functions are equivalent to those strongly aggregating metrics on products.

	%
	\section{Quasi-pseudometric aggregation functions}\label{sec2}
	%
	
	We begin by recalling that a quasi-metric \(d\) on a nonempty set \(X\) satisfies all the classic axioms of a metric, except for the symmetry axiom. Additionally, the separation axiom is modified to state that 
	
	\[
	d(x,y) = d(y,x) = 0 \text{ if and only if } x = y
	\]
	for all \(x, y \in X\). If this weaker separation axiom is removed, we then obtain a quasi-pseudometric.
	
	Next, we will provide a precise definition of the central subject of this paper: quasi-pseudometric aggregation functions.
	
	\begin{definition}[\mbox{see \cite[Definition 1]{MassaVale13}}]\label{def:eqpmaf}
		A function $F: [0,+\infty)^I\rightarrow [0,+\infty)$ is said to be 
		\begin{itemize}
			\item a \textbf{quasi-pseudometric aggregation function on products} if whenever $\big\{(X_i,d_i):i\in I\big\}$ is a family of quasi-pseudometric spaces, then $F\circ d_\Pi$ is a quasi-pseudometric on $\prod_{i\in I} X_i$ where $d_\Pi:\left(\prod_{i\in I} X_i\right)\times \left(\prod_{i\in I} X_i\right)\to [0,+\infty]^I$ is given by
			$$d_\Pi(x,y)=\big(d_i(x_i,y_i)\big)_{i\in I}$$
			for every $x,y\in \prod_{i\in I} X_i;$
			\item a \textbf{quasi-pseudometric aggregation function on sets} if whenever $\{d_{i}:i \in I\}$ is a family of quasi-pseudometrics on a nonempty set $X$, then $F \circ d_\Delta$ is a quasi-pseudometric on $X$, where $d_\Delta: X \times X \rightarrow [0,+\infty]^{I}$ is given by $$d_\Delta(x,y)=\big(d_{i}(x,y)\big)_{i \in I}$$
			for every $x,y \in X.$
		\end{itemize}
		
		Similar definitions can be stated for quasi-metrics, pseudometrics and metrics.
	\end{definition}
	
	\begin{remark}\label{rem:qpmaf}
		As we have previously noted, quasi-pseudometric aggregation functions and their variants have been examined extensively in existing literature:
		\begin{itemize}
			\item metric aggregation functions on products were mainly examined by Dobo\v{s} \cite{Dobos98} under the name metric preserving functions (see also \cite{Corazza99}); 
			\item pseudometric aggregation functions on products and on sets were characterized by Pradera and Trillas \cite{PraTri02}; 
			\item quasi-metric aggregation functions on products were studied by Mayor and Valero  \cite{MayorVale10} under the name asymmetric distance aggregation function;
			\item metric aggregation functions on sets were analyzed by Mayor and Valero  \cite{MayorVale19} under the name metric preserving functions; 
			\item quasi-metric aggregation functions on sets were considered by Mi\~nana and Valero  \cite{MinyaVale19} under the name $n$-quasi-metric aggregation function.
			\item quasi-pseudometric aggregation functions on products and on sets were characterized in \cite{FBRL25a}.
		\end{itemize}

	\end{remark}

	We next collect some characterizations about the aggregation of quasi-pseudometrics from the previous references. But first, we introduce some terminology. We denote by $0_I$ the element of $[0,+\infty)^I$ where all coordinates are zero. By an abuse of notation, the symbol \( \leq \) will also represent the componentwise order on \( [0, +\infty)^I \), which is induced by the usual order \( \leq \) on \( [0, +\infty) \). 
	
	Given $a,b,c\in [0,+\infty)^I,$ we say that $(a,b,c)$ is a triangle triplet if $a\leq b+c, $ $b\leq a+c,$ and $c\leq a+b$ \cite{Dobos98}.
	
	Additionally, recall that a function \( F: [0, +\infty)^I \rightarrow [0, +\infty) \) is defined to be subadditive if it satisfies the condition \( F(a+b) \leq F(a) + F(b) \) for all \( a, b \in [0, +\infty)^I \). Here, the operation on the left-hand side of the inequality refers to the sum defined componentwise.
	
	\begin{theorem}[\mbox{\cite[Chapter 9, Theorem 2]{Dobos98}}]\label{thm:Dobos}
		Let $F: [0,+\infty)^I\rightarrow [0,+\infty)$. The follo\-wing statements are equivalent:
		\begin{enumerate}[(1)]
			\item $F$ is a metric aggregation function on products;
			\item $F^{-1}(0)=\{0_I\}$ and $\big(F(a),F(b),F(c)\big)$ is a triangle triplet whenever $(a,b,c)$ is a triangle triplet.
		\end{enumerate}
	\end{theorem}
	
	\begin{theorem}[\mbox{\cite{PraTri02}}]\label{thm:PraTri}
		Let $F: [0,+\infty)^I\rightarrow [0,+\infty)$. The follo\-wing statements are equivalent:
		\begin{enumerate}[(1)]
			\item $F$ is a pseudometric aggregation function on products;
			\item $F$ is a pseudometric aggregation function on sets;
			\item $F(0_I)=0$ and $\big(F(a),F(b),F(c)\big)$ is a triangle triplet whenever $(a,b,c)$ is a triangle triplet.
		\end{enumerate}
	\end{theorem}

	\begin{theorem}[\mbox{\cite[Theorems 6 and 9]{MayorVale10}}]\label{thm:MayorValero}
		Let $F: [0,+\infty)^I\rightarrow [0,+\infty)$. The follo\-wing statements are equivalent:
		\begin{enumerate}[(1)]
			\item $F$ is a quasi-metric aggregation function on products;
			\item $F^{-1}(0)=\{0_I\}$ and $F(a)\leq F(b)+F(c),$ whenever $a\leq b+c,$ $a,b,c\in [0,+\infty)^I;$
			\item $F^{-1}(0)=\{0_I\},$ $F$ is subadditive and monotone.
		\end{enumerate}
	\end{theorem}

	\begin{theorem}[\mbox{\cite[Theorem 3.16]{FBRL25a}}]\label{thm:FBRL}
		Let $F: [0,+\infty)^I\rightarrow [0,+\infty)$. The follo\-wing statements are equivalent:
		\begin{enumerate}[(1)]
			\item $F$ is a quasi-pseudometric aggregation function on products;
			\item $F$ is a quasi-pseudometric aggregation function on sets;
			\item $F\left(0_I\right)=0$, $F$ is subadditive and monotone;
			\item $F\left(0_I\right)=0$ and $F(a)\leq F(b)+F(c),$ whenever $a\leq b+c,$ $a,b,c\in [0,+\infty)^I.$
		\end{enumerate}
	\end{theorem}
	%
	\section{Strongly quasi-pseudometric aggregation functions on products}\label{sec3}
	%
	
	In \cite{BorsikDobos81b,Dobos98}, the authors posed and solved the following question concerning the aggregation of metrics on products. 
	Let $F : [0,+\infty)^I \longrightarrow [0,+\infty)$
	be a metric aggregation function on products. If $\big\{(X_i,d_i) : i \in I\big\}$ is a family of metric spaces, then $F \circ d_\Pi$ defines a metric on the Cartesian product $\prod_{i \in I} X_i$. The question is: under which conditions is the product topology 
	\[
	\prod_{i \in I} \mathscr{T}(d_i)
	\]
	metrizable by $F \circ d_\Pi$, that is, when do we have
	\[
	\prod_{i \in I} \mathscr{T}(d_i) \;=\; \mathscr{T}(F \circ d_\Pi) \, ?
	\]

	When $|I|=1$, the problem was first investigated by Sreenivasan \cite{Sree47}, and later Bors\'ik and Dobo\v{s} \cite{BorsikDobos81a} completely characterized the functions 
	$F : [0,+\infty) \longrightarrow [0,+\infty)$
	for which every metric $d$ is topologically equivalent to $F \circ d$, showing that these are precisely the continuous ones (see also \cite{Corazza99}). Such functions were called \emph{strongly metric-preserving}, and we retain the adjective ``strongly'' here. For an arbitrary index set $I$, the general solution was given in \cite{BorsikDobos81b} (see also \cite{Dobos98}).
	
	We next study the same question for quasi-pseudometric aggregation functions on products. 
	
	\begin{definition}
		A (quasi-)(pseudo)metric aggregation function on products $F:[0,+\infty)^I$$\to [0,+\infty)$ is said to be a \textbf{strongly (quasi-)(pseudo)metric aggregation function on products} if for every family of (quasi-)(pseudo)metric spaces $\big\{(X_i,d_i):i\in I\big\}$ then
		$$\prod_{i \in I} \mathscr{T}(d_i) \;=\; \mathscr{T}(F \circ d_\Pi) .$$
	\end{definition}

	Now, we shall illustrate the above definition with some examples.
	\begin{example}
		Let $I$ be a finite set and let us consider $M\colon [0,+\infty)^I \to [0,+\infty),$ given by
		\[
		M\big(a\big)=\max_{i\in I} a_i.
		\]
		for every $a=(a_i)_{i\in I}\in [0,+\infty)^I. $
		Let us show that $M$ is a strongly \mbox{(quasi-)}\mbox{(pseudo
			)}metric aggregation function on products. An easy verification proves that $M^{-1}(0)=\{0_I\}$, $M$ is subadditive and monotone. Then by Theorems \ref{thm:Dobos}, \ref{thm:PraTri}, \ref{thm:MayorValero}, \ref{thm:FBRL}, it is a (quasi-)(pseudo)metric aggregation function on products. 
		
		Next, we demonstrate that $M$ is also strongly.\\
		
		Let $\big\{(X_i,d_i):i\in I\big\}$ be a family of (quasi-)(pseudo)metric spaces, $x\in \prod_{i\in I} X_i$ and $\varepsilon>0$. Then 
		\begin{align*}
			B_{M\circ d_\Pi}(x,\varepsilon)
			&=\left\{y\in\prod_{i\in I} X_i : M\circ d_\Pi(x,y)<\varepsilon\right\}=
			\left\{y\in\prod_{i\in I} X_i : \max_{i\in I} d_i(x_i,y_i)<\varepsilon\right\}\\
			&=\prod_{i\in I} B_{d_i}(x_i,\varepsilon)
		\end{align*}
		which is an open set in the product topology.
		Hence $\mathscr T(M\circ d_\Pi)\subseteq \prod_{i\in I} \mathscr T(d_i)$.
		
		To check the other inclusion, let $\prod_{i\in I} B_{d_i}(x_i,\varepsilon_i)$ be a basic open set in the product topology $\prod_{i\in I} \mathscr T(d_i)$, where $x_i\in X_i$, $\varepsilon_i>0,$ for all $i\in I.$ 
		
		Then
		\[
		B_{M\circ d_\Pi}(x,\varepsilon)
		=
		\prod_{i\in I} B_{d_i}(x_i,\varepsilon)
		\subseteq
		\prod_{i\in I}  B_{d_i}(x_i,\varepsilon_i)
		\]
		where $\varepsilon:=\min_{i\in I}\varepsilon_i>0.$
		Therefore, every basic open set in the product topology is open in
		$\mathscr T(M\circ d_\Pi)$, and hence
		\[
		\prod_{i\in I} \mathscr T(d_i)\subseteq \mathscr T(M\circ d).
		\]
		Thus, $M$ is a strongly (quasi-)(pseudo)metric aggregation function.

		Notice that the classical supremum metric $d_\infty$ on $\mathbb{R}^n$ is nothing else but $M\circ d_\Pi$ for the family of metric spaces $\big\{(\mathbb{R},d_i):i\in\{1,\ldots,n\}\big\}$, where $d_i$ is the Euclidean metric for every $i\in\{1,\ldots,n\}.$
		
	\end{example}

	\begin{example}
		Let $F:[0,+\infty)^\mathbb{N}\to [0,+\infty)$ given by
		$$F(a)=\sum_{n=1}^\infty \frac{1}{2^n}\frac{a_n}{1+a_n}$$
		for all $a=(a_n)_{n\in\mathbb{N}}\in [0,+\infty)^\mathbb{N}.$
		
		It is well known that $F$ is a metric aggregation function on products (see \cite[Theorem 22.3]{Willard}).
	\end{example}

	The proof of the following result is analogous to that of \cite[Chapter 10, Lemma 1]{Dobos98} concerning metric aggregation functions on products. We include it here for completeness.
	
	\begin{proposition}[\mbox{compare with \cite[Chapter 10, Lemma 1]{Dobos98}}]\label{prop:strongqmaf}
		Let $F: [0,+\infty)^I \to [0,+\infty)^I$ be a (quasi-)metric aggregation function on products. 
		If $\big\{(X_i,d_i):i \in I\big\}$ is a family of (quasi-)metric spaces, then
		\[
		\prod_{i \in I} \mathscr{T}(d_i) \subseteq \mathscr{T}(F\circ d_\Pi).
		\]
	\end{proposition}
	
	\begin{proof}
		Let $j\in I,$ $x_j\in X_j$, $\varepsilon>0$ and $\pi_j^{-1}\big(B_{d_j}(x_j,\varepsilon)\big)$ be a subbasic open set in the product topology $\prod_{i \in I} \mathscr{T}(d_i).$ Consider $x \in \pi_j^{-1}(B_{d_j}(x_j,\varepsilon)).$
		
		Define $e \in [0,+\infty)^I$ as
		\[
		e_i =
		\begin{cases}
			0 & \text{if } i \neq j, \\
			2\varepsilon & \text{if } i = j,
		\end{cases}
		\]
		for all $i\in I.$
		By Theorem \ref{thm:Dobos} (or Theorem \ref{thm:MayorValero}, in the quasi-metric case), $\delta:=\frac{F(e)}{2}>0$. 
		
		Let us check that
		\[
		x \in B_{F\circ d_\Pi}(x,\delta) \subseteq \pi_j^{-1}\big(B_{d_j}(x_j,\varepsilon)\big).
		\]
		
		Suppose that $y \in B_{F\circ d_\Pi}(x,\delta)$, that is,
		\[
		F\circ d_\Pi(x,y) = F\big( (d_i(x_i, y_i))_{i \in I} \big) < \delta.
		\]
		If $y \notin \pi_j^{-1} \big(B_{d_j}(x_j, \varepsilon)\big)$, then
		\[
		d_j(x_j, y_j) > \varepsilon.
		\]
		Hence $(e,d_\Pi(x,y),d_\Pi(x,y))$ is a triangle triplet
		so by  Theorem \ref{thm:Dobos} (or Theorem \ref{thm:MayorValero}, in the quasi-metric case)
		
		\[
		2\delta=F(e) \leq 2 F(d_\Pi(x,y)), 
		\]
		which contradicts $y \in B_{F\circ d_\Pi}(x,\delta).$ Therefore, $y \notin \pi_j^{-1} \big(B_{d_j}(x_j, \varepsilon)\big)$ which finishes the proof.	
	\end{proof}
	
	\begin{remark}	
		The previous results does not hold for (quasi-)pseudometric aggregation functions on products.
		Notice that $F : [0,+\infty)^I \to [0,+\infty]$ given by $F(x) = 0$ for all $x\in [0,+\infty)^I$, is a (quasi-)pseudometric aggregation function on products. In fact, for any family of (quasi-)pseudometric spaces $\big\{(X_i, d_i):i\in I\big\},$ it follows that  $F\circ d_\Pi$ is the indiscrete pseudometric on $\prod_{i\in I} X_i$. Consequently, $\mathscr{T}(F\circ d_\Pi)$ is the indiscrete topology. Thus, we can conclude
		\[
		\prod_{i \in I} \mathscr{T}(d_i) \not\subseteq \mathscr{T}(F\circ d_\Pi)
		\]
		unless $\prod_{i \in I} \mathscr{T}(d_i)$ is itself the indiscrete topology. 
	\end{remark}
	
	To characterize when the inclusion of Proposition \ref{prop:strongqmaf} holds for quasi-pseudometric aggregation functions on products, it is essential the set $F^{-1}(0),$ as demonstrated by the following result.

	\begin{proposition}\label{prop:sqpmafp_prod_inclusion}
		Let $F : [0,+\infty)^I \to [0,+\infty)$ be a (quasi-)pseudometric aggregation function on products. Then $\prod_{i \in I} \mathscr{T}(d_i) \subseteq \mathscr{T}(F\circ d_\Pi),$ for every family $\big\{(X_i, d_i) : i \in I\big\}$ of (quasi-)pseudometric spaces, if and only if 
		$F^{-1}(0)=\{0_I\}.$
	\end{proposition}
	
	\begin{proof}
		For proving the necessary condition, we proceed by contradiction. Assume that there exists $a\in [0,+\infty)^I$ such that $F(a)=0$ and $a\neq 0_I.$ Then, we can find $j \in I$ such that $a_j\neq 0.$ Now, let $b\in [0,+\infty)^I$ be defined as follows:
		\[b_i=\begin{cases}
			0&\text{ if }i\neq j,\\
			a_j&\text{ if }i=j,
		\end{cases}\]
		for all $i\in I.$ Then $(b,a,a)$ is a triangle triplet. By Theorem \ref{thm:PraTri}  (Theorem \ref{thm:FBRL} in the quasi-pseudometric case), $F(b)\leq 2 F(a)=0$ so $F(b)=0.$
		
		Let $X$ be a set with at least two different elements. Given $i\in I\backslash\{j\}$, let $d_i$ be the indiscrete pseudometric. Define $d_j$ by
		\[
		d_j(\xi,\eta) =
		\begin{cases}
			a_j & \text{if } \xi \neq \eta, \\
			0 & \text{if } \xi = \eta,
		\end{cases}
		\]
		for all $\xi,\eta\in X.$ 
		Then $\big\{(X, d_i) : i \in I\big\}$ is a family of pseudometric spaces. Moreover, the product topology  $\prod_{i \in I} \mathscr{T}(d_i)$ on $X^I$ is not the indiscrete topology. Indeed, for any $\xi \in X$, the singleton $\{\xi\}$ is $\mathscr{T}(d_j)$-open and, therefore, $\pi_j^{-1}(\{\xi\}) \in \prod_{i \in I} \mathscr{T}(d_i)$. Since $\pi_j^{-1}(\{\xi\})$ is neither empty nor equal to $X^I$, the product topology is not indiscrete.
		
		Nevertheless, given $x,y\in X^I$
		\[
		F\circ d_\Pi(x,y) = F\big((d_i(x_i,y_i))_{i \in I}\big) =\begin{cases} F(b) = 0 & \text{ if  } x_j \neq y_j,\\
			F(0)=0&\text{ if }x_j=y_j.\end{cases}
		\]
		So $\mathscr{T}(F\circ d_\Pi)$ is the indiscrete topology. Then
		\[
		\prod_{i \in I} \mathscr{T}(d_i) \nsubseteq \mathscr{T}(F\circ d_\Pi),
		\]
		which is a contradiction.

		The sufficiency condition follows as the proof of Proposition \ref{prop:strongqmaf}
	\end{proof}

	As a consequence of our previous results together with Theorems \ref{thm:Dobos}, \ref{thm:MayorValero}, \ref{thm:FBRL}, \ref{thm:PraTri}, we deduce the following.
	
	\begin{theorem}\label{thm:strongqmafp_prod_in_ag}
		Let $F : [0,+\infty)^I \to [0,+\infty)$ be a (quasi-)pseudometric aggregation function on products. The following statements are equivalent:
		\begin{enumerate}[(1)]
			\item $F$ is a (quasi-)metric aggregation function on products;
			\item $F^{-1}(0)=\{0_I\};$
			\item  $\prod_{i \in I} \mathscr{T}(d_i) \subseteq \mathscr{T}(F\circ d_\Pi),$ for every family $\big\{(X_i, d_i) : i \in I\big\}$ of (quasi-)pseudometric spaces.
		\end{enumerate}
	\end{theorem}
	
	So far, we have considered the problem of characterizing the inclusion $\prod_{i \in I} \mathscr{T}(d_i) \subseteq \mathscr{T}(F\circ d_\Pi)$. We now turn our attention to establishing results for the reverse inclusion.

	\begin{theorem}\label{thm:strongqpmafp_ag_in_prod}
		Let $F : [0,+\infty)^I \to [0,+\infty)$ be a (quasi-)(pseudo)metric aggregation function on products. Then $\mathscr{T}(F\circ d_\Pi) \subseteq \prod_{i \in I} \mathscr{T}(d_i),$ for every family $\big\{(X_i,d_i) : i \in I\big\}$ of (quasi-)(pseudo)metric spaces,  if and only if $F$ is continuous at $0_I$.
	\end{theorem}
	
	\begin{proof}
		
		To prove sufficiency, let us consider the family of (quasi-)(pseudo)metric spaces $\big\{(\mathbb{R},d_e):i\in I\big\}$, where $d_e$ is the Euclidean metric. By assumption $\mathscr{T}(F\circ d_\Pi) \subseteq \prod_{i \in I} \mathscr{T}(d_e).$ Hence given $\varepsilon>0$ we can find $\delta>0$ and a finite subset $J$ of $I$ such that
		\[\bigcap_{j\in J}\pi_j^{-1}\big(B_{d_e}(0,\delta)\big)\subseteq B_{F\circ d_\Pi}(0_I,\varepsilon).\] 
		Let $a\in [0,+\infty)^I$ such that $0\leq a_j<\delta$ for all $j\in J.$ Then $d_e(0,a_j)=a_j<\delta$ for all $j\in J$ so $a\in \bigcap_{j\in J}\pi_j^{-1}\big(B_{d_e}(0,\delta)\big).$ Hence 
		\[(F\circ d_\Pi)(0,a)=F\Big(\big(d_e(0,a_i)\big)_{i\in I}\Big)=F(a)=|F(a)-F(0_I)|<\varepsilon\]
		so $F$ is continuous at $0_I.$
		
		Conversely, let $x \in \prod_{i \in I} X_i$ and $\varepsilon > 0$. Since $F$ is continuous at $0_I$, we can find $\delta > 0$ and $J \subset I$ finite such that:
		\[
		\text{if } a \in [0,+\infty)^I, \quad 0 \leq a_j <\delta \ \forall j \in J \Rightarrow |F(a) - F(0_I)| = F(a) < \varepsilon.
		\]
		
		Then
		\[
		x \in \bigcap_{j \in J} \pi_j^{-1}\big(B_{d_j}(x_j, \delta)\big) \subseteq B_{F\circ d_\Pi}(x, \varepsilon).
		\]
		
		In fact, if $y \in \bigcap_{j \in J} \pi_j^{-1}\big(B_{d_j}(x_j, \delta)\big)$ then $d_j(x_j, y_j) < \delta$ for all $j \in J$ and
		\[
		F\circ d_\Pi(x,y) = F(d_i(x_i,y_i))_{i \in I} < \varepsilon.
		\]
	\end{proof}
	
	From Proposition \ref{prop:strongqmaf} and Theorem \ref{thm:strongqpmafp_ag_in_prod}, we obtain the following characterization of strongly (quasi-)metric aggregation functions on products.

	As an immediate consequence of Theorems \ref{thm:strongqmafp_prod_in_ag} and \ref{thm:strongqpmafp_ag_in_prod},  we can characterize strongly (quasi-)pseudometric aggregation functions as follows.
	
	\begin{theorem}\label{thm:strongcharac}
		Let $F : [0,+\infty)^I \to [0,+\infty)$ be a (quasi-)pseudometric aggregation function on products. The following statements are equivalent:
		\begin{enumerate}[(1)]
			\item $F$ is a strongly (quasi-)pseudometric aggregation function on products;
			\item $F$ is a strongly (quasi-)metric aggregation function on products;
			\item $F^{-1}(0)=\{0_I\}$ and $F$ is continuous at $0_I.$
		\end{enumerate}
	\end{theorem}

	\begin{theorem}\label{thm:ChactStronglyProd}
		Given a function $F : [0,+\infty)^I \to [0,+\infty)$, consider the following statements:
		\begin{enumerate}[(1)]
			\item $F$ is a strongly quasi-metric aggregation function on products;
			\item $F$ is a strongly quasi-pseudometric aggregation function on products;
			\item $F^{-1}(0)=\{0_I\},$ $F$ is subadditive, monotone  and continuous at $0_I;$
			\item $F$ is a strongly metric aggregation function on products;
			\item $F$ is a strongly pseudometric aggregation function on products;
			\item $F^{-1}(0)=\{0_I\},$ $F$ is continuous at $0_I$ and $\big(F(a),F(b),F(c)\big)$ is a triangle triplet whenever $(a,b,c)$ is a triangle triplet.
		\end{enumerate}
		Then
		$$\text{(1)} \Leftrightarrow \text{(2)}   \Leftrightarrow \text{(3)}  \Longrightarrow \text{(4)} \Leftrightarrow \text{(5)} \Leftrightarrow \text{(6)}.$$
	\end{theorem}
	
	\begin{proof}
		(1) $\Rightarrow$ (2) 
		Assume that $F$ is a strongly quasi-metric aggregation function on products. 
		By Theorems \ref{thm:MayorValero} and \ref{thm:FBRL}, it follows that $F$ is a quasi-pseudometric aggregation function on products. 
		The claim then follows directly from Theorem~\ref{thm:strongcharac}.

		(2) $\Leftrightarrow$ (3) This is an immediate consequence of Theorems \ref{thm:FBRL} and \ref{thm:strongcharac}.

		(3) $\Rightarrow$ (4) 
		A straightforward computation shows that subadditivity together with monotonicity ensures the preservation of triangle triplets. 
		Hence, by Theorem \ref{thm:MayorValero}, $F$ is a metric aggregation function on products. 
		Moreover, Theorem \ref{thm:PraTri} guarantees that $F$ is also a pseudometric aggregation function on products. 
		The desired conclusion follows from Theorem~\ref{thm:strongcharac}.

		(4) $\Rightarrow$ (5) 
		By applying Theorems~\ref{thm:MayorValero} and~\ref{thm:PraTri}, we obtain that $F$ is a pseudometric aggregation function on products. 
		The implication then follows by Theorem~\ref{thm:strongcharac}.

		(5) $\Leftrightarrow$ (6) This equivalence follows directly from Theorems \ref{thm:PraTri} and \ref{thm:strongcharac}.
		
	\end{proof}

	The following example demonstrates that not all statements in the above theorem are equivalent.
	
	\begin{example}[\mbox{\cite[Chapter 5, Example 1]{Dobos98}}]
		
		Let $F:[0,+\infty) \rightarrow[0,+\infty)$ be given as
		$$
		F(x)= \begin{cases}x & \text { if } x \leq 2, \\ 1+\frac{1}{x-1} & \text { if } x>2 ,\end{cases}
		$$
		for all $x\in [0,+\infty).$ It is shown in \cite{Dobos98} that $F$ is a metric aggregation function (on products). Moreover, by Theorem \ref{thm:strongcharac}, it is a strongly metric aggregation function on products. Nevertheless, $F$ is not monotone so it is not a strongly quasi-(pseudo)metric aggregation function on products.

	\end{example}
	
	\begin{remark}\label{rem:Iuncountable}
		It is well-known (see, for example, \cite[Theorem 22.3]{Willard}) that the product topology of the Cartesian product of a family of metric spaces is metrizable if and only if all but a countable number of these metric spaces are single points.
		
		This is also true for (quasi-)(pseudo)metric spaces since they are first countable. 
		
		Then if $|I|>\aleph_0,$ there does not exist any function $F:[0,+\infty)^I\to [0,+\infty)$ that satisfies the statements of Theorem \ref{thm:ChactStronglyProd}. 
	\end{remark}
	
	Figure \ref{fig:products} summarizes the relationships among the families of aggregation functions on products considered so far. Each region corresponds to a family of aggregation functions, and its label indicates the underlying structure being aggregated.
	
	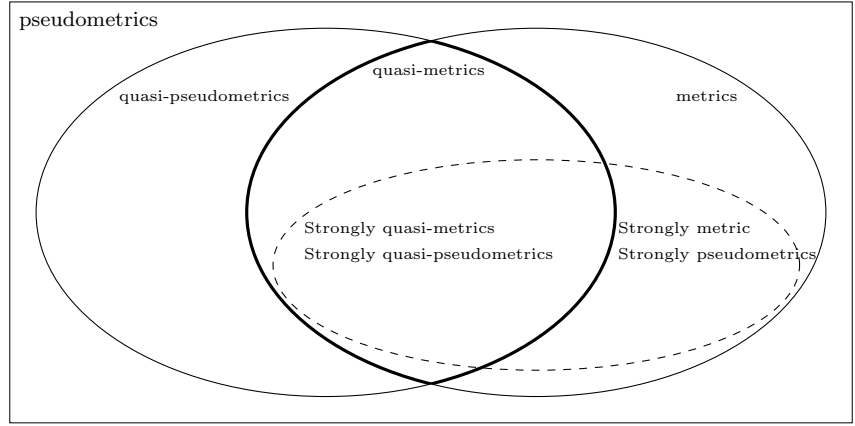
\begin{figure}[htbp]
		\centering
		
		\begin{tikzpicture}[font=\footnotesize,scale=0.7]
			
			\draw[black,line width=0.01pt] (0,0) rectangle (16,8);
			\node[anchor=north west] at (0,8) {pseudometrics};
			
			\draw (6,4) ellipse (5.5cm and 3.5cm);
			\node[anchor=north west] at (1.9,6.5) {{\tiny quasi-pseudometrics}};
			
			\draw (10,4) ellipse (5.5cm and 3.5cm);
			\node[anchor=north east] at (14,6.5) {{\tiny metrics}};
			
			\begin{scope}
				\clip (6,4) ellipse (5.5cm and 3.5cm);
				\draw[very thick] (10,4) ellipse (5.5cm and 3.5cm);
			\end{scope}
			
			\begin{scope}
				\clip (10,4) ellipse (5.5cm and 3.5cm);
				\draw[very thick] (6,4) ellipse (5.5cm and 3.5cm);
			\end{scope}
			
			
			\draw[dashed] (10,3) ellipse (5cm and 2cm);
			\node[anchor=north east] at (9.2,7) {{\tiny quasi-metrics}};
			
			\node[anchor=north east,align=left] at (10.5,4)
			{{\tiny Strongly quasi-metrics }\\
				{\tiny Strongly quasi-pseudometrics}};
			
			\node[anchor=north east,align=left] at (15.5,4)
			{{\tiny Strongly metric }\\
				{\tiny Strongly pseudometrics}};
			
		\end{tikzpicture}
		\caption{Relations between aggregation functions on products}\label{fig:products}
	\end{figure}

	%
	\section{Strongly quasi-pseudometric aggregation functions on sets}\label{sec4}
	%
	
	We now turn to the question of how topology is preserved by quasi-pseudometric aggregation functions on sets. To clarify, we first define what we mean by topology preservation in this context. In the case of quasi-pseudometric aggregation functions on products, we have identified the conditions under which the aggregated quasi-pseudometric is compatible with the product topology. However, for quasi-pseudometric aggregation functions on sets, we must consider a topology defined on a single fixed set instead of a Cartesian product. As noted in \cite{PRL21}, the supremum topology is a natural and suitable candidate for this purpose.

	More precisely, let \( X \) denote a nonempty set and  and $\{d_i:i\in I\}$ a family of quasi-pseudometrics  on $X.$ Consider the mapping $\mathcal{I}:X\to X^I$ given by $\mathcal I(x)=(x)_{i\in I}$, for all $x\in X.$ When $X$ is endowed with the supremum topology $\bigvee_{i \in I} \mathscr{T}(d_i)$  and $X^I$ with the product topology, the map \( \mathcal{I} \) becomes a topological embedding  of $X$ into the diagonal $\Delta=\{x\in X^I: x_j=x_k\text{ for all }j,k\in I\}.$   This observation leads us to the following definition.
	
	\begin{definition}
		A (quasi-)(pseudo)metric aggregation function on sets $F:[0,+\infty)^I$$\to [0,+\infty)$ is said to be a \textbf{strongly (quasi-)(pseudo)metric aggregation function on sets} if for every nonempty set $X$ and every family $\{d_i:i\in I\}$ of (quasi-)(pseudo)metrics on $X$ then
		$$\bigvee_{i \in I} \mathscr{T}(d_i)=\mathscr{T}(F\circ d_\Delta).$$
	\end{definition}
	
	We next observe the following straightforward relationship between strongly quasi-pseudometric aggregation functions on products and on sets.
	
	\begin{proposition}\label{prop:products_implies_sets}
		If $F:[0,+\infty)^I$$\to [0,+\infty)$ is a strongly (quasi-)(pseudo)metric aggregation function on products then it is a strongly (quasi-)(pseudo)metric aggregation function on sets.
	\end{proposition}
	
	\begin{proof}
		Let $X$ be a set  and let $\{d_i:i\in I\}$ be a family of (quasi-)(pseudo)metrics on $X.$ By assumption the product topology $\prod_{i\in I} \mathscr{T}(d_i)$ on $X^I$ coincides with $\mathscr{T}(F\circ d_\Pi).$ Then the restriction of $\prod_{i\in I} \mathscr{T}(d_i)$ to $\Delta$ is isomorphic to $\bigvee_{i \in I} \mathscr{T}(d_i).$ Since $\mathcal I$ is an isomorphism from $(X,\mathscr{T}(F\circ d_\Delta))$ onto $(\Delta,\mathscr{T}(F\circ d_\Pi)|_\Delta)$, the proof follows.
	\end{proof}
	
	The following example shows that, when dealing with (quasi-)(pseudo)metric aggregation functions on sets, one cannot obtain a result analogous to Proposition \ref{prop:strongqmaf}.
	
	\begin{example}\label{example:diffsets}
		Let $P_2:[0,+\infty)^2\to [0,+\infty)$ be the projection over the second coordinate. Obviously, $P_2$ is a (quasi-)(pseudo)metric aggregation function on sets.
		
		On the real line $\mathbb{R}$, consider the discrete metric $d_1$ and the Euclidean metric $d_2$ on $\mathbb{R}.$ Then 
		$$\mathscr{T}(d_1)\bigvee \mathscr{T}
		(d_2)=\mathscr{T}(d_1)\not\subseteq \mathscr{T}(P_2\circ d_\Delta)=\mathscr{T}(d_2).$$
	\end{example}
	
	In this context, an analogous for Proposition \ref{prop:strongqmaf} is the following result which resembles Proposition \ref{prop:sqpmafp_prod_inclusion}.
	
	\begin{proposition}\label{prop:sqpmafs_sup_in_agg}
		Let $F : [0,+\infty)^I \to [0,+\infty)$ be a quasi-pseudometric aggregation function on sets. Then $\bigvee_{i \in I} \mathscr{T}(d_i) \subseteq \mathscr{T}(F\circ d_\Delta),$ for every family $\big\{(X, d_i) : i \in I\big\}$ of quasi-pseudometric spaces, if and only if 
		$F^{-1}(0)=\{0_I\}.$
		
		The statement remains valid if the term ``quasi-pseudometric'' is replaced by ``quasi-metric'' or ``pseudometric''.
	\end{proposition}
	
	\begin{proof} We prove the necessary condition by contradiction. So suppose that we can find $a\in [0,+\infty)^I$ verifying $F(a)=0$ but there exists $j \in I$ with $a_j \neq 0$.  
		
		Let $X = \{p,q\}$ be a set with two elements, and for each $i \in I$, consider
		\[
		d_i = a_i \cdot d_D,
		\]
		where $d_D$ is the discrete metric on $X$.  
		
		Then $\big\{(X,d_i): i \in I\big\}$ is a family of pseudometric spaces, so
		\[
		\bigvee_{i\in I} \mathscr{T}(d_i) \subseteq \mathscr{T}(F\circ d_\Delta).
		\]
		Let $\varepsilon = \tfrac{a_j}{2}>0$.  
		Then there exists $\delta_\varepsilon > 0$ such that
		\[
		B_{F\circ d_\Delta}(p, \delta_\varepsilon) \subseteq B_{d_j}(p,\varepsilon).
		\]
		
		Since
		\[
		F\circ d_\Delta(p,q) = F\big((d_i(p,q))_{i \in I}\big) = F\big((a_i)_{i \in I}\big) = 0,
		\]
		then $q \in B_{F\circ d_\Delta}(p,\delta_\varepsilon) \subseteq B_{d_j}(p,\varepsilon)$, that is
		\[
		d_j(p,q) = a_j\cdot d_D(p,q) = a_j < \varepsilon = \tfrac{a_j}{2},
		\]
		which is a contradiction. Hence, $a=0_I.$
		
		The previous proof applies to quasi-pseudometric and pseudometric aggregation functions on sets. 
		For quasi-metric aggregation functions on sets, let $J = \{i \in I : a_i = 0\}$, and let $X = \{p,q\}$ as above. 
		Define, for each $j \in J$,
		\[
		d_j(x,y) =
		\begin{cases}
			0, & \text{if } x = y \text{ or } (x,y)=(p,q),\\
			1, & \text{otherwise}.
		\end{cases}
		\]
		for every $x,y\in X.$ For $i \in I \setminus J$, set $d_i = a_i \cdot d_D$ as before. 
		The proof then proceeds exactly as in the previous case.

		The sufficiency follows from a straightforward adaptation of the proof of Proposition \ref{prop:strongqmaf}, together with Theorems \ref{thm:PraTri}, \ref{thm:FBRL}.
	\end{proof}
	
	It is natural to wonder whether the previous Proposition also holds for metric aggre\-gation functions on sets. The answer is no as the next example demonstrates (compare with \cite[Example 7]{MayorVale19}).
	
	\begin{example}\label{example:smafs}
		Let $F:[0,+\infty)^2\to [0,+\infty)$ be defined as
		$$F(a)=\begin{cases}
			0&\text{ if }a_1\cdot a_2=0,\\
			1&\text{ if }a_1\neq 0,a_2\neq 0.		
		\end{cases}$$
		for every $a=(a_1,a_2)\in [0,+\infty)^2.$ Given a nonempty set $X$ and two metrics $\{d_1,d_2\}$ on $X$ then 
		$$(F\circ d_\Delta)(x,y)=\begin{cases}
			1&\text{ if }x\neq y,\\
			0&\text{ if }x=y,
		\end{cases}$$
		for all $x,y\in X.$ Hence $F\circ d_\Delta$ is the discrete metric. Therefore, $F$ is a metric aggregation function on sets. Moreover, since $\mathscr{T}(F\circ d_\Delta)$ is the discrete topology then 
		$$\mathscr{T}(d_1)\bigvee \mathscr{T}(d_2)\subseteq \mathscr{T}(F\circ d_\Delta).$$
		However, $F^{-1}(0)=\big\{(a_1,a_2)\in\mathbb{R}^2:a_1\cdot a_2=0\big\}\neq \big\{(0,0)\big\}.$
		
	\end{example}
	
	Therefore, Proposition \ref{prop:sqpmafs_sup_in_agg} is not true for metric aggregation functions on sets. The primary reason for this is that if we have a family of metrics  $\{d_i : i \in I\}$ on a nonempty set $X$, 
	then the image $\mathrm{Im}d_\Delta$ of the function $d_\Delta$ is contained within the set
	$\{0_I\} \cup (0,+\infty)^I.$ As a result, the values of $F$ that fall outside this range are irrelevant (they can be 0).
	Therefore, a different characterization is necessary, which will require the following concept.

	\begin{definition}[\mbox{\cite[Definition 6.2.4]{BeerBook}}] A multifunction $h: X \rightrightarrows Y$ between two topological spaces is said to be \textbf{upper semicontinuous} at $x \in X$ if for every open subset $G$ of $Y$ containing $h(x)$ we can find an open set $O$ containing $x$ such that $h(o) \subseteq G$ for every $o \in O$.
\end{definition}

\begin{proposition}\label{prop:necessary_strongmaf}
	Let $F: [0,+\infty)^I \to [0,+\infty)$ be a {\rm (}quasi-{\rm )}{\rm (}pseudo{\rm )}metric aggregation function on sets. Endow $[0,+\infty)$ with the left-order topology having as base all intervals of the form $[0,\varepsilon)$, $\varepsilon>0.$ Let $X$ be a nonempty set and $\{d_i:i\in I\}$ be a family of {\rm (}quasi-{\rm )}{\rm (}pseudo{\rm )}metrics on $X$.  If  $\bigvee_{i \in I} \mathscr{T}(d_i) \subseteq \mathscr{T}(F\circ d_\Delta)$ then $F|_{d_\Delta x}^{-1}$ is upper semicontinuous at $0$ for every $x\in X,$ where  $F|_{d_\Delta x}$ denotes the restriction of $F$ to $\mathrm{Im}\big(d_\Delta(x,\cdot)\big)$.
\end{proposition}

\begin{proof}
	Suppose that  $\bigvee_{i \in I} \mathscr{T}(d_i) \subseteq \mathscr{T}(F\circ d_\Delta)$. Given $x\in X$,  let $V$ be an open set in the topology induced on $\mathrm{Im}\big(d_\Delta(x,\cdot)\big)$ by the product topology, that contains $F|_{d_\Delta x}^{-1}(0).$ Since $0_I\in F|_{d_\Delta x}^{-1}(0)\subseteq V$ we can find a basic open set in the induced product topology $G=\bigcap_{j\in J} \pi_j^{-1}
	\big([0,\varepsilon_j)\big)\cap \mathrm{Im}\big(d_\Delta(x,\cdot)\big),$ where $J\subseteq I$ is finite and $\varepsilon_j>0$ for all $j\in J,$  such that
	$0_I\in G\subseteq V.$
	By assumption there exists $\delta>0$ such that
	\[B_{F \circ d_\Delta}\left(x, \delta\right) \subseteq \bigcap_{j \in J} B_{d_j}\left(x, \varepsilon_j\right).\] 
	
	Let $\alpha\in [0,\delta)$. If $F|_{d_\Delta x}^{-1}(\alpha)=\varnothing$ then trivially $F|_{d_\Delta x}^{-1}(\alpha)\subseteq V.$ Otherwise, let $a\in F|_{d_\Delta x}^{-1}(\alpha).$ Then there exists $y_\alpha\in X$ such that $d_\Delta (x,y_\alpha)=\big(d_i(x,y_\alpha)\big)_{i\in I}=a.$ Then $F\big(d_\Delta(x,y_\alpha)\big)=F(a)=\alpha<\delta$ so $y_\alpha\in B_{F \circ d_\Delta}\left(x, \delta\right).$ Hence, $y_\alpha\in \bigcap_{j \in J} B_{d_j}\left(x, \varepsilon_j\right)$ so $d_j(x,y_\alpha)<\varepsilon_j$ for all $j\in J.$ Consequently, $a\in G\subseteq V.$ By the arbitrariness of $a\in F|_{d_\Delta x}^{-1}(\alpha),$ it follows that 
	$$F|_{d_\Delta x}^{-1}(\alpha)\subseteq G\subseteq V,$$ which shows that $F|_{d_\Delta x}^{-1}$ is upper semicontinuous at 0. 
\end{proof}

At this moment, it is natural to wonder whether we can find a nonempty set $X$ and a family of (quasi-)(pseudo)metrics $\{d_i:i\in I\}$ on $X$ such that $\mathrm{Im}\big(d_\Delta(x,\cdot)\big)=[0,+\infty)^I,$ for some $x\in X.$ The following results provide a positive answer in all cases except that of metrics. 

\begin{lemma}\label{lemma:ImQPM}
	There exist a family of {\rm(}quasi-{\rm)}pseudometric spaces $\big\{ (X,d_{i}) : i \in I\big\}$ and $x \in X$ such that
	$$\operatorname{Im}\big(d_\Delta(x,\cdot)\big)= [0,+\infty)^{I}.$$
\end{lemma}

\begin{proof}
	Let us consider the family $\big\{(\mathbb{R}^{I}, d_{i}) : i \in I \big\},$ where 
	$$d_{i}(x, y) = d_{e}(x_{i}, y_{i}) \text{ for every } x, y \in \mathbb{R}^{I},$$
	with \(d_{e}\) denoting the Euclidean metric on \(\mathbb{R}\). It is clear that \(d_{i}\) is a \mbox{(quasi-)}pseudometric.
	
	Now, given \(a \in [0, +\infty)^{I}\), we find that 
	$$d_{\Delta}(0_{I}, a) = (d_{i}(0, a_{i}))_{i \in I} = (a_{i})_{i \in I}.$$
	Thus, we have proven our claim.
	
\end{proof}

\begin{lemma}\label{lemma:ImQM}
	There exist a family of quasi-metric spaces $\big\{ (X,d_{i}) : i \in I\big\}$ and $x \in X$ such that
	$$\operatorname{Im}\big(d_\Delta(x,\cdot)\big)= [0,1)^{I}.$$
\end{lemma}

\begin{proof}
	Let us consider the following quasi-metrics on $\mathbb{R}:$
	$$l(x,y):=\max \lbrace x-y, 0 \rbrace,\hspace*{0.5cm}
	u(x,y):= \max \lbrace y-x, 0 \rbrace,$$
	for every $x,y\in\mathbb{R}.$
	Now, let \(\{ (\mathbb{R}^I, d_{i}) : i \in I \}\) be defined by
	$$d_{i}(x,y)=\sup\Big\{\big\{ l(x_{j},y_{j}) : j \in I-\lbrace i \rbrace \big\} \cup \big\{ u(x_{i},y_{i})\big\}\Big\} \wedge  1$$
	for every $x,y\in\mathbb{R}^I.$
	Given $a \in [0,1)^{I}$ then
	$$d_{i}(0_I,a)=a_{i}.$$
	This concludes the proof.
\end{proof}

\begin{problem}
	It remains an open question whether the previous lemmas hold for a family of metrics on a fixed set.
\end{problem}

From Proposition \ref{prop:necessary_strongmaf} and the previous lemmas we deduce the following.

\begin{corollary}
	Let $F: [0,+\infty)^I \to [0,+\infty)$ be a {\rm (}quasi-{\rm )}pseudometric aggregation function on sets. Endow $[0,+\infty)$ with the left-order topology having as base all intervals of the form $[0,\varepsilon)$, $\varepsilon>0.$ If $\bigvee_{i \in I} \mathscr{T}(d_i) \subseteq \mathscr{T}(F\circ d_\Delta),$ for every family of (quasi-)pseudometric spaces $\big\{(X,d_i):i\in I\big\},$ then $F^{-1}$ is upper semicontinuous at $0.$
	
	The result is also valid for quasi-metric aggregation functions on sets.
\end{corollary}

For the converse of Proposition \ref{prop:necessary_strongmaf}, because of Proposition \ref{prop:sqpmafs_sup_in_agg} and Example \ref{example:smafs}, we focus on functions that aggregate metrics on sets.

\begin{proposition}\label{prop:sufficient_strongmaf}
	Let $F: [0,+\infty)^I \to [0,+\infty)$ be a metric aggregation function on sets. Endow $[0,+\infty)$ with the left-order topology. Let $X$ be a nonempty set and $\{d_i:i\in I\}$ be a family of metrics on $X$.  If  $F|_{d_\Delta x}^{-1}$ is upper semicontinuous at $0$ for every $x\in X,$ then  $\bigvee_{i \in I} \mathscr{T}(d_i) \subseteq \mathscr{T}(F\circ d_\Delta).$
\end{proposition}

\begin{proof} 
	It is easy to check that  $\mathrm{Im} d_\Delta\subseteq \{0_I\}\cup (0,+\infty)^I$ so by \cite[Theorem 12]{MayorVale19} we deduce that $F|_{d_\Delta x}^{-1}(0)=\{0_I\}.$
	
	Let $J$ be a finite subset of $I$ and let $\bigcap_{j\in J} B_{d_j}(x,\varepsilon_j)$ be a basic open set in $\bigvee_{i \in I} \mathscr{T}(d_i),$ where $x\in X$ and $\varepsilon_j>0$ for all $j\in J.$

	Since 
	$$\{0_I\}=F|_{d_\Delta x}^{-1}(0)\subseteq \left( \bigcap_{j\in J}\pi_j^{-1}\big([0,\varepsilon_j)\big)\right)\bigcap \mathrm{Im}\big(d_\Delta(x,\cdot)\big)$$
	and  $F|_{d_\Delta x}^{-1}$ is upper semicontinuous at $0$ we can find $\delta>0$ such that if $\beta\in [0,\delta)$ then 
	$$F|_{d_\Delta x}^{-1}(\beta)\subseteq \left(\bigcap_{j\in J}\pi_j^{-1}\big([0,\varepsilon_j)\big)\right)\bigcap \mathrm{Im}\big(d_\Delta(x,\cdot)\big).$$
	
	Suppose that $y\in B_{F\circ d_\Delta}(x,\delta),$ that is, $F\big((d_i(x,y))_{i\in I}\big)<\delta$. Then $\big(d_i(x,y)\big)_{i\in I} \in F|_{d_\Delta x}^{-1}\Big(F\big((d_i(x,y))_{i\in I}\big)\Big)$ so
	$\big(d_i(x,y)\big)_{i\in I}\in \bigcap_{j\in J} \pi_j^{-1}\big([0,\varepsilon_j)\big)\cap \mathrm{Im}\big(d_\Delta(x,\cdot)\big).$ That is  $d_j(x,y)<\varepsilon_j$ for all $j\in J.$ This proves 
	$$B_{F\circ d_\Delta}(x,\delta)\subseteq \bigcap_{j\in J} B_{d_j}(x,\varepsilon_j).$$
\end{proof}

From Propositions \ref{prop:necessary_strongmaf} and \ref{prop:sufficient_strongmaf}, it is immediately deduced the following.

\begin{theorem}\label{thm:strongmafs_sup_in_ag}
	Let $F: [0,+\infty)^I \to [0,+\infty)$ be a metric aggregation function on sets. Endow $[0,+\infty)$ with the left order topology having as base all intervals of the form $[0,\varepsilon)$, $\varepsilon>0.$ Let $X$ be a nonempty set and $\{d_i:i\in I\}$ be a family of metrics on $X$.  Then  $\bigvee_{i \in I} \mathscr{T}(d_i) \subseteq \mathscr{T}(F\circ d_\Delta)$ if and only if $F|_{d_\Delta x}^{-1}$ is upper semicontinuous at $0$ for every $x\in X.$ 
\end{theorem}

\begin{remark}
	Notice that the metric aggregation function on sets presented in Exam\-ple \ref{example:diffsets}, does not satisfy the hypothesis of Proposition \ref{prop:sufficient_strongmaf}. 
	Concretely, notice that for every $x\in\mathbb{R}$, $\mathrm{Im}(d_\Delta(x,\cdot))=\{(0,0)\}\cup (\{1\}\times (0,+\infty)).$ Then $P_2^{-1}(0)\cap \mathrm{Im}(d_\Delta(x,\cdot))=\{(0,0)\}.$ Thus, $\{(0,0)\}$ is open in the induced product topology and it contains $P_2|_{d_\Delta x}^{-1}(0).$ Nevertheless, for every open set $[0,\delta)$ that contains $0$, we have: 
	$$P_2|_{d_\Delta x}^{-1}\left(\tfrac{\delta}{2}\right)=\left(\mathbb{R}\times\left\{\tfrac{\delta}{2}\right\}\right)\cap \big(\{(0,0)\}\cup (\{1\}\times (0,+\infty))\big)=\left\{\left(1,\tfrac{\delta}{2}\right)\right\}\not\subseteq \{(0,0)\}.$$
	Therefore,   $P_2|_{d_\Delta x}^{-1}$ is not upper semicontinuous at $0$. However, observe that $P_2^{-1}$ is upper semicontinuous at $0$. 
\end{remark}

We will next address the problem of characterizing those (quasi-)(pseudo)metric aggregation functions on sets \( F \) for which \( \mathscr{T}(F \circ d_\Delta) \subseteq \bigvee_{i \in I} \mathscr{T}(d_i) \).

\begin{theorem}\label{thm:strongqpmafs_ag_in_sup}
	Let $F : [0,+\infty)^I \to [0,+\infty)$ be a {\rm(}quasi-{\rm)}pseudometric aggregation function on sets.  Endow $[0,+\infty)$ with the left order topology. Then, $\mathscr{T}(F\circ d_\Delta) \subseteq \bigvee_{i \in I} \mathscr{T}(d_i),$ for every family $\big\{d_i : i \in I\big\}$ of {\rm(}quasi-{\rm)}pseudometrics on a nonempty set $X$,  if and only if $F$ is continuous at $0_I.$
	
	The result also holds for quasi-metric aggregation functions on sets.
\end{theorem}

\begin{proof}
	To demonstrate the necessary condition for a (quasi-)pseudometric aggregation function on sets, we consider the family of (quasi-)pseudometrics spaces $\big\{(\mathbb{R}^I,d_i):i\in I\big\}$ introduced in the proof of Lemma \ref{lemma:ImQPM}.
	Let $\varepsilon>0.$ By hypothesis, there exist a finite subset $J$ of $I$ and $\delta_j>0$ for all $j\in J$ such that
	$$\bigcap_{j\in J} B_{d_j}(0_I,\delta_j)\subseteq B_{F\circ d_\Delta}(0_I,\varepsilon).$$
	Let $a\in \bigcap_{j\in J} \pi_j^{-1}\big([0,\delta_j)\big),$ which is an open set in the product topology containing $0_I.$ Since $\mathrm{Im}(d_\Delta(0_I,\cdot)) =[0,+\infty)^I,$ then there exists $y\in \mathbb{R}^I$ such that $d_\Delta (0_I,y)=a.$ Hence $d_j(0_I,y)=a_j<\delta_j$ for all $j\in J,$ which implies that 
	$$y\in \bigcap_{j\in J} B_{d_j}(0_I,\delta_j)\subseteq B_{F\circ d_\Delta}(0_I,\varepsilon).$$
	This leads to $F(a)=F\big(d_\Delta (0_I,y)\big)=(F\circ d_\Delta)(0_I,y)<\varepsilon$, so $F$ is continuous at $0_I.$
	
	In the case of a quasi-metric aggregation function on sets, to prove necessity, we consider the family of quasi-metric spaces constructed in the proof of Lemma \ref{lemma:ImQM}. We then proceed as above, noting that $[0,1)^I$ is a neighborhood of $0_I$.
	
	For proving sufficiency, let $\big\{d_i : i \in I\big\}$ be a family of quasi-pseudometrics on a nonempty set $X.$ Consider a basic open set $B_{F\circ d_\Delta}(x,\varepsilon)$ in $\mathscr{T}(F\circ d_\Delta)$, where $x\in X$ and $\varepsilon>0.$ Since $F$ is continuous at $0_I$ there exist a finite subset $J$ of $I$ and $\delta_j>0$ for all $j\in J$ such that whenever $a\in \bigcap_{j\in J} \pi_j^{-1}\big([0,\delta_j)\big),$ then $F(a)<\varepsilon.$
	
	A straightforward verification shows that
	$$\bigcap_{j\in J} B_{d_j}(x,\delta_j)\subseteq B_{F\circ d_\Delta}(x,\varepsilon),$$
	thus proving that $B_{F\circ d_\Delta}(x,\varepsilon)$ is open in $\bigvee_{i\in I}\mathscr{T}(d_i).$
\end{proof}

In the absence of a result analogous to lemmas \ref{lemma:ImQPM} and \ref{lemma:ImQM} for metrics, in the case of metric aggregation function on sets, the characterization of the inclusion $\mathscr{T}(F\circ d_\Delta) \subseteq \bigvee_{i \in I} \mathscr{T}(d_i)$ becomes less elegant for metric aggregation functions on sets.

\begin{theorem}\label{thm:strongmafs_ag_in_sup}
	Let $F : [0,+\infty)^I \to [0,+\infty)$ be a metric aggregation function on sets.  Endow $[0,+\infty)$ with the left order topology. Then $\mathscr{T}(F\circ d_\Delta) \subseteq \bigvee_{i \in I} \mathscr{T}(d_i)$ for every family $\big\{d_i : i \in I\big\}$ of metrics on a nonempty set $X,$  if and only if $F|_{d_\Delta x}$ is continuous at $0_I$ for all $x\in X,$  where  $F|_{d_\Delta x}$ denotes the restriction of $F$ to $\mathrm{Im}\big(d_\Delta(x,\cdot)\big)$.
\end{theorem}

\begin{proof}
	To show the necessary condition, fix $x\in X$ and let $\varepsilon>0.$ By hypothesis, there exist a finite subset $J$ of $I$ and positive real numbers $\delta_j>0$ for all $j\in J$ such that
	$$\bigcap_{j\in J} B_{d_j}(x,\delta_j)\subseteq B_{F\circ d_\Delta}(x,\varepsilon).$$
	Let $a\in \bigcap_{j\in J} \pi_j^{-1}\big([0,\delta_j)\big)\cap \mathrm{Im}\big(d_\Delta(x,\cdot)\big),$ which is an open set in the induced product topology containing $0_I.$ Then there exists $y\in X$ such that $d_\Delta (x,y)=a.$ Hence $d_j(x,y)=a_j<\delta_j$ for all $j\in J$ so 
	$$y\in \bigcap_{j\in J} B_{d_j}(x,\delta_j)\subseteq B_{F\circ d_\Delta}(x,\varepsilon).$$
	Then $F|_{d_\Delta x}(a)=F\big(d_\Delta (x,y)\big)=(F\circ d_\Delta)(x,y)<\varepsilon$, so $F|_{d_\Delta x}$ is continuous at $0_I.$
	
	For proving sufficiency, consider a basic open set $B_{F\circ d_\Delta}(x,\varepsilon)$ in $\mathscr{T}(F\circ d_\Delta)$, where $x\in X$ and $\varepsilon>0.$ Since $F|_{d_\Delta x}$ is continuous at $0_I$ there exist a finite subset $J$ of $I$ and $\delta_j>0$ for all $j\in J$ such that whenever $a\in \bigcap_{j\in J} \pi_j^{-1}\big([0,\delta_j)\big)\cap \mathrm{Im}\big(d_\Delta(x,\cdot)\big),$ then $F|_{d_\Delta x}(a)<\varepsilon.$
	
	A straightforward verification leads to 
	$$\bigcap_{j\in J} B_{d_j}(x,\delta_j)\subseteq B_{F\circ d_\Delta}(x,\varepsilon)$$
	proving that $B_{F\circ d_\Delta}(x,\varepsilon)$ is open in $\bigvee_{i\in I}\mathscr{T}(d_i).$
\end{proof}

The following description of strongly (quasi-)pseudometric or quasi-metric aggregation functions on sets is a consequence of Propositions \ref{prop:products_implies_sets}, \ref{prop:sqpmafs_sup_in_agg} and Theorems \ref{thm:strongcharac}, \ref{thm:strongmafs_sup_in_ag}, \ref{thm:strongqpmafs_ag_in_sup}.

\begin{theorem}\label{thm:products=sets}
	Given a {\rm(}quasi-{\rm)}pseudometric aggregation function on sets $F : [0,+\infty)^I \to [0,+\infty),$ the following statements are equivalent:
	\begin{enumerate}[(1)]
		\item $F$ is a strongly {\rm(}quasi-{\rm)}pseudometric aggregation function on products;
		\item $F$ is a strongly {\rm(}quasi-{\rm)}pseudometric aggregation function on sets;
		\item $F^{-1}(0)=\{0_I\}$ and $F$ is continuous at $0_I.$
	\end{enumerate}

	The statement remains valid if the term ``quasi-pseudometric'' is replaced by ``quasi-metric''.
\end{theorem}

For strongly metric aggregation functions on sets, we can provide the following characterization coming from Theorems \ref{thm:strongmafs_sup_in_ag} and \ref{thm:strongmafs_ag_in_sup}.

\begin{theorem}\label{thm:smafs}
	Given a metric aggregation function on sets $F : [0,+\infty)^I \to [0,+\infty),$ the following statements are equivalent:
	\begin{enumerate}[(1)]
		\item $F$ is a strongly metric aggregation function on sets;
		\item for every family $\big\{d_i : i \in I\big\}$ of metrics on a nonempty set $X,$ 
		\begin{enumerate}
			\item $F|_{d_\Delta x}^{-1}$ is upper semicontinuous at $0$ for every $x\in X;$ 
			\item  $F|_{d_\Delta x}$ is continuous at $0_I$ for all $x\in X;$
		\end{enumerate}
		where  $F|_{d_\Delta x}$ denotes the restriction of $F$ to $\mathrm{Im}\big(d_\Delta(x,\cdot)\big)$.
	\end{enumerate}
	
\end{theorem}

From our previous results and those of \cite{MayorVale19,MinyaVale19} we obtain the following:

\begin{theorem}\label{thm:ChactStronglySets}
	Given a function $F : [0,+\infty)^I \to [0,+\infty)$, consider the following statements:
	\begin{enumerate}[(1)]
		\item $F$ is a strongly quasi-metric aggregation function on products;
		\item $F$ is a strongly quasi-metric aggregation function on sets;
		\item $F$ is a strongly quasi-pseudometric aggregation function on products;
		\item $F$ is a strongly quasi-pseudometric aggregation function on sets;
		\item $F^{-1}(0)=\{0_I\},$ $F$ is subadditive, monotone  and continuous at $0_I;$
		\item $F$ is a strongly pseudometric aggregation function on products;
		\item $F$ is a strongly pseudometric aggregation function on sets;
		\item $F^{-1}(0)=\{0_I\},$ $F$ is continuous at $0_I$ and $\big(F(a),F(b),F(c)\big)$ is a triangle triplet whenever $(a,b,c)$ 
		is a triangle triplet;
		\item $F$ is a strongly metric aggregation function on sets;
		\item $F^{-1}(0)=\{0_I\},$ $\big(F(a),F(b),F(c)\big)$ is a triangle triplet whenever $(a,b,c)$ 
		is a triangle triplet, and  for every family $\big\{d_i : i \in I\big\}$ of metrics on a nonempty set $X,$ 
		\begin{enumerate}
			\item $F|_{d_\Delta x}^{-1}$ is upper semicontinuous at $0$ for every $x\in X;$ 
			\item  $F|_{d_\Delta x}$ is continuous at $0_I$ for all $x\in X;$
		\end{enumerate}
		where  $F|_{d_\Delta x}$ denotes the restriction of $F$ to $\mathrm{Im}\big(d_\Delta(x,\cdot)\big)$.
	\end{enumerate}
	
	Then
	$$\text{(1)} \Leftrightarrow \text{(2)}   \Leftrightarrow \text{(3)} \Leftrightarrow  \text{(4)}  \Leftrightarrow  \text{(5)} \Longrightarrow \text{(6)} \Leftrightarrow \text{(7)} \Leftrightarrow \text{(8)} \Longrightarrow \text{(9)} \Leftrightarrow \text{(10)}.$$
\end{theorem}

\begin{remark}
	Using the theory of uniform spaces (see, for instance, \cite{KelleyBook}), we show that when the index set 
	$I$ is uncountable, it is impossible to construct a function $F:[0,+\infty)^I\to [0,+\infty)$ satisfying any of the assertions of the preceding theorem.
	
	Let $(X,\mathscr{T})$ be a completely regular topological space. It is well known that such a space is uniformizable, that is, there exists a uniformity $\mathcal{U}$ on $X$ whose induced topology $\mathscr{T}(\mathcal{U})$ coincides with $\mathscr{T}.$
	
	Moreover, a classical result states that a uniform space $(X,\mathcal{U})$ is pseudometri\-zable if and only if the uniformity $\mathcal{U}$ has a countable base \cite[Chapter 6, Theorem 13]{KelleyBook}. Consequently, not every uniformity is pseudometrizable. 
	
	Recall also that for any uniformity $\mathcal{U}$ on $X$ there exists a family of pseudometrics $\{d_i:i\in I\}$ on $X,$ called the \emph{gauge} of $\mathcal{U},$ such that $\mathcal{U}=\bigvee_{i\in I}\mathcal{U}_{d_i},$ where $\mathcal{U}_{d_i}$ denotes the uniformity generated by the pseudometric $d_i.$ Thus 
	$$\mathscr{T}(\mathcal{U})=\bigvee_{i\in I}\mathscr{T}(\mathcal{U}_{d_i})=\bigvee_{i\in I}\mathscr{T}(d_i).$$ 
	
	Now suppose that $|I|>\aleph_0$ and that $F:[0,+\infty)^I\to [0,+\infty)$ is a strongly quasi-pseudometric aggregation function on sets. Then $F$ is a strongly pseudometric aggregation function on sets. Let $(X,\mathcal{U})$ be a non-pseudometrizable uniform space, and let $\{d_i:i\in I\}$ be its gauge. For instance, if $(Y,\mathcal{U})$ is a non-trivial uniform space and $Z$ is uncountable, one may consider the product uniformity on $Y^Z.$ This uniformity induces the topology of pointwise convergence  and, by reasoning as in Remark \ref{rem:Iuncountable}, it is not pseudometrizable.
	
	It then follows that $F\circ d_\Delta$ is a pseudometric on $X$ whose induced topology coincides with the supremum topology $\bigvee_{i\in I}\mathscr{T}(d_i)=\mathscr{T}(\mathcal{U}),$ which contradicts the assumption that $\mathcal{U}$ is not pseudometrizable.
	
\end{remark}

\begin{problem}
	We do not know whether strongly metric aggregation functions on sets coincide with those on products.
\end{problem}


\begin{thebibliography}{10}
	
	\bibitem{ArnauCalabuigGonzalezSP24}
	R.~Arnau, J.~M. Calabuig, Á. González, and E.~A. Sánchez-Pérez,
	\emph{Moduli of continuity in metric models and extension of livability
		indices}, Axioms \textbf{13} (2024), no.~3, Paper No. 192.
	
	\bibitem{BeerBook}
	G.~Beer, \emph{Topologies on closed and closed convex sets}, vol. 268, Kluwer
	Academic Publishers, 1993.
	
	\bibitem{BorsikDobos81a}
	J.~Bors\'{\i}k and J.~Dobo\v{s}, \emph{Functions whose composition with every
		metric is a metric}, Math. Slovaca \textbf{31} (1981), no.~1, 3--12.
	
	\bibitem{BorsikDobos81b}
	J.~Bors\'{\i}k and J.~Dobo\v{s}, \emph{On a product of metric spaces}, Math. Slovaca \textbf{31}
	(1981), no.~2, 193--205.
	
	\bibitem{Corazza99}
	P.~Corazza, \emph{Introduction to metric-preserving functions}, Amer. Math.
	Monthly \textbf{106} (1999), no.~4, 309--323.
	
	\bibitem{Demir23}
	I.~Demir, \emph{Novel correlation coefficients for interval-valued fermatean
		hesitant fuzzy sets with pattern recognition application}, Turkish Journal of
	Mathematics \textbf{47} (2023), no.~1, 213--233.
	
	\bibitem{Dobos98}
	J.~Dobo\v{s}, \emph{Metric preserving functions}, \v Stroffek, Ko\v sice, 1998.
	
	\bibitem{FBRL25a}
	A.~Fructuoso-Bonet and J.~Rodr\'iguez-L\'opez, \emph{Aggregation functions as
		lax morphisms of quantales}, Fuzzy Sets and Systems \textbf{513} (2025),
	Paper No. 109395.
	
	\bibitem{KelleyBook}
	J.~L. Kelley, \emph{General topology}, Springer, 1955.
	
	\bibitem{MassaVale13}
	S.~Massanet and O.~Valero, \emph{On aggregation of metric structures: the
		extended quasi-metric case}, Int. J. Comput. Intelligent Systems \textbf{6}
	(2013), 115--126.
	
	\bibitem{MayorVale10}
	G.~Mayor and O.~Valero, \emph{Aggregation of asymmetric distances in {C}omputer
		{S}cience}, Information Sciences \textbf{180} (2010), 803--812.
	
	\bibitem{MayorVale19}
	G.~Mayor and O.~Valero, \emph{Metric aggregation functions revisited}, European J. Combin.
	\textbf{80} (2019), 390--400.
	
	\bibitem{MinyaVale19}
	J.-J. Mi{\~{n}}ana and O.~Valero, \emph{Characterizing quasi-metric aggregation
		functions}, Int. J. Gen. Syst. \textbf{48} (2019), no.~8, 890--909.
	
	\bibitem{Pap15}
	E.~Pap, \emph{Aggregation functions as a base for decision making}, Synthesis
	2015 - International Scientific Conference of IT and Business-Related
	Research, 2015, pp.~143--146.
	
	\bibitem{PRL21}
	T.~Pedraza and J.~Rodríguez-López, \emph{New results on the aggregation of
		norms}, Mathematics \textbf{9} (2021), no.~18, 2291.
	
	\bibitem{PraTri02}
	A.~Pradera and E.~Trillas, \emph{A note on pseudometrics aggregation}, Int. J.
	Gen. Syst. \textbf{31} (2002), no.~1, 41--51.
	
	\bibitem{LHSR22}
	A.~F. {Roldán López de Hierro}, M.~Sánchez, and C.~Roldán,
	\emph{Multi-criteria decision making involving uncertain information via
		fuzzy ranking and fuzzy aggregation functions}, Journal of Computational and
	Applied Mathematics \textbf{404} (2022), 113138.
	
	\bibitem{Sree47}
	T.~K. Sreenivasan, \emph{Some properties of distance functions}, J. Indian
	Math. Soc. (N.S.) \textbf{11} (1947), 38--43.
	
	\bibitem{BookTorraNaru}
	V.~Torra and Y.~Narukawa, \emph{Modeling decisions. information fusion and
		aggregation operators}, Springer Berlin, Heidelberg, 2017.
	
	\bibitem{Willard}
	S.~Willard, \emph{General topology}, Dover, 1970.
	
	\bibitem{Wilson35}
	W.~A. Wilson, \emph{On certain types of continuous transformations of metric
		spaces}, Amer. J. Math. \textbf{57} (1935), no.~1, 62--68.
	
\end{thebibliography}

\end{document}